\documentclass{amsart}
\usepackage{graphicx} % Required for inserting images
\usepackage{inputenc}
\usepackage{graphicx}
\usepackage{subcaption} % immagini
\usepackage{amsmath,mathtools}
\usepackage{amsfonts}

\usepackage{amssymb}
\usepackage{mathrsfs}
\usepackage{color}
\usepackage{enumerate}
\usepackage[colorlinks]{hyperref}
\usepackage{algorithm}
\usepackage{mdframed}
\usepackage{algorithmic}
\usepackage{siunitx}
\usepackage{multirow}
\usepackage{stackrel}
\usepackage{booktabs}
\usepackage{tabularx}
\usepackage{arydshln}
\usepackage{cite}
\usepackage{url}
\usepackage{comment}
\usepackage{listings}

\usepackage{booktabs}

\usepackage[top=2.8cm, bottom=2.8cm, left=2.6cm, right=2.6cm]{geometry}

\newtheorem{theorem}{Theorem}

\newtheorem{corollary}{Corollary}
\newtheorem{lemma}[theorem]{Lemma}

\theoremstyle{definition}

\theoremstyle{remark}
\usepackage{enumerate}

\newtheorem{Remark}{Remark}
\newtheorem{Assumption}{Assumption}

\begin{document}
                            
\title{Block diagonal preconditioners for arrow-shaped multiple saddle-point linear systems}
\author{Luca Bergamaschi and Michele Bergamaschi}
%\AuthorNames{L. Bergamaschi \and M. Bergamasch
\date{}
\footnotetext[1]{Department of Civil Environmental and Architectural Engineering, University of Padua, Via Marzolo, 9, 35100 Padua, Italy, \texttt{E-mail}: {luca.bergamaschi@unipd.it}}
\footnotetext[2]{\texttt{E-mail}: michele15456@gmail.com}

\begin{abstract}{We consider multiple saddle-point linear systems with block-arrow structure. We analyze exact and inexact Schur complement based block diagonal preconditioner, by characterizing the eigenvalues of the preconditioned matrices
in terms of the zeros of specific rational functions. Numerical results onto large tests coming from liquid crystal modeling and PDE-constrained optimization validate the obtained eigenvalue bounds.} \end{abstract}
\maketitle

%\section{Introduction}

\newcommand{\bB}{\mathbf B}
\newcommand{\bR}{\mathbf R}
\newcommand{\bBt}{\mathbf {\widetilde{B}}}
\newcommand{\bRt}{\mathbf {\widetilde{R}}}

\newcommand{\bE}{\mathbf E}
\newcommand{\bA}{\mathbf A}
\newcommand{\fz}{z}
\newcommand{\fu}{u}

\section{Introduction}
  We consider the iterative solution of a block  multiple saddle-point linear system $\mathcal{A}x = b$, where 
\begin{equation}
\label{SPmat}
\mathcal{A}=\left[\begin{array}{ccccc}
    A_0 & B_1^\top & B_2^\top & \dots & B_N^\top \\
    B_1 & -A_1 &  &  & \\
    B_2&   & -A_2 &  & \\
    \vdots &  &  & \ddots & \\
    B_N &  &  &  & -A_N
\end{array}\right].
\end{equation}
We assume that $A_0\in\mathbb{R}^{n_0\times n_0}$ is symmetric positive definite, all other square block matrices $A_k\in\mathbb{R}^{n_k\times n_k}$ are symmetric positive semi-definite and $B_k\in\mathbb{R}^{n_k\times n_{0}}$ have full rank (for $k=1,\dots, N$). 
We will assume throughout the paper that
\begin{Assumption}
\label{ass1}
\[ n_0 \ge \sum_{k=1}^N n_k.\]
\end{Assumption}
Note that this is a necessary condition for invertibility of $\mathcal A$ in the particular case $A_1 = A_2 = \ldots = A_N = 0.$ This result extends trivially the result in  \cite[Lemma 1]{BradleyGreif2026} to multiple saddle-point matrices.
A  sufficient condition for invertibility of $\mathcal A$ is obtained by considering it as as saddle-point matrix 
\begin{equation} \label{regsp}
 \mathcal  A = \begin{bmatrix} A_0 & \mathbf B^\top \\ \mathbf B & -\mathbf A_+\end{bmatrix}, \qquad \text{with} \ \ 
\mathbf B = \begin{bmatrix} B_1^\top & \ldots & B_N^\top\end{bmatrix}^\top, \ \mathbf A_+ = \text{blkdiag}(A_1, \ldots, A_N).
\end{equation}
As well known, a sufficient condition for invertibility is $A_0$ SPD and
$\bB$ with full row rank \cite[Theorem 3.1]{bgl05}. Translated to our multiple saddle-point matrix, this condition can be written as: 
\begin{Assumption}
\label{ass2}
 \[\bigcap _{k=1}^N \text{range} (B_k^\top) = \emptyset.\]
 \end{Assumption}
The position of nonzero blocks in \eqref{SPmat} accounts for the definition of \textit{arrow-shaped matrix}, or simply arrow matrix,
which we will use in the sequel.
Matrices of the form \eqref{SPmat} arise from several applications such as liquid crystal director modeling, interior point methods in constrained optimization, the Stokes-Darcy equation, poroelasticity, and magma-mantle dynamics, to mention a few. Examples of these problems can be found
in \cite{Szyld, greif2026, ramage2013preconditioned, iter:BGZ, BenziLC, FerFraJanCasTch19, BeikBenzi2022, Rhebergen, Antonietti2021}.
%{Denoting with  
%$S_1 = A_1 + B_1 A_0^{-1} B_1^\top$ and $S_k = A_k + %B_k{S}^{-1}_{k-1}B_k^\top$ (for $k = 2, \ldots, N$) the %Schur complements, % we finally assume
%that 
%\begin{equation} \label{ass0} \lambda_{\min}(A_k) \cdot \lambda_{\min}(B_k{S}^{-1}_{k-1}B_k^\top) \ne 0, \quad k = 1, \ldots, N.\end{equation} 
%these conditions ensure that $S_k$,
%$k = 1, \ldots, N$ are invertible,
%We assume also that $n_k\leq n_{k-1}$ for all $k$. 
%which also guarantees invertibility of $\mathcal{A}$ %\cite{Bradley,BEIK2024403,pearson2023symmetric}. }
Multiple saddle-point linear systems more commonly take a block tridiagonal form and have been recently analyzed, with a focus on preconditioning in \cite{bradley2023eigenvalue, Balani-et-al-2023a, Balani-et-al-2023b, BerFerMar25, Bergamaschi-COAP-2024}, for double saddle-point and, more generally in \cite{SZ, LBMB2026,  BBP_NLAA_2026, pilotto-et-al2026}, for multiple saddle-point matrices. A theoretical study on the preconditioned matrices having the form \eqref{SPmat} is still lacking. 
The focus of this work is the development
and the related spectral analysis of block diagonal Schur complement based preconditioners.

We will also briefly consider arrow-structured multiple saddle-point matrices with the "arrow" appearing in the last row and column. Even if such matrices can be symmetrically permuted into a block tridiagonal
form, maintaining the arrow block structure can have some advantages
in specific problems. We then adapt the previously  developed theory to predict the eigenvalue distribution of these block-diagonally preconditioned  matrices.

Considering SPD block diagonal preconditioners, we fulfill the assumptions that make the MINRES~\cite{minres} method applicable to our linear systems. The MINRES algorithm can explicitly exploit the symmetry of $\mathcal{A}$. As a consequence, MINRES features a three-term recurrence relation, which is beneficial for its implementation (low memory requirements because subspace bases need not be stored) and its purely eigenvalue-based convergence analysis (via the famous connection to orthogonal polynomials; see~\cite{fischer1996polynomial,greenbaum1997iterative}). 

The structure of the paper is as follows: in Section \ref{sec:Schur} we give an explicit form of the block Schur complements of the arrow matrices and prove that they are all SPD under the sole condition
that $\mathcal A$ is invertible. In Section \ref{sec:BDP} we present some (Schur complement based) block diagonal preconditioners, while Section \ref{sec:eig} is devoted to relate the eigenvalue distribution of
the preconditioned matrices, to the zeros of suitable polynomial/rational functions and validate the 
obtained bounds onto a synthetic multiple $5 \times 5$ saddle-point problem. Section \ref{sec:double}
is devoted to  check the bounds onto double saddle-point linear systems among which  the liquid crystal director modeling solved up to $10^7$ unknowns. In Section \ref{sec:downright} we take advantage of the previous analysis to describe the extremal eigenvalues of the preconditioned matrix in case of a "downright arrow" matrix structure and
validate the obtained bounds onto a realistic PDE-constrained optimization problem up to very large size.
Section \ref{sec:conc} provides some concluding remarks and future perspectives.

\begin{comment}

By contrast, GMRES needs to store subspace bases and its convergence analysis is, in general, dependent on the corresponding eigenspaces as well, which are more complicated to analyze than eigenvalues (see, e.g., \cite{embree2022descriptive}).

{
  We consider the iterative solution of a block tridiagonal multiple saddle-point linear system $\mathcal{A}x = b$, where 
\begin{equation}
\label{SPmat}
\mathcal{A}=\left[\begin{array}{ccccc}
    A_0 & B_1^T & B_2^T & \dots & B^T \\
    B_1 & -A_1 &  &  & \\
    B_2&   & -A_2 &  & \\
    \vdots &  &  & \ddots & \\
    B_N &  &  &  & -A_N
\end{array}\right]
\end{equation}
\end{comment}

\section{Schur complements: explicit expression and properties}
\label{sec:Schur}
We will first give theoretical expression of
all the Schur complements of matrix $\mathcal A$.
Next we will prove that all of them are SPD under
the sole condition that $\mathcal A$ is invertible. 
Matrix $\mathcal A$ admits the LDL factorization, namely it can be written as
\[ \mathcal A = \mathcal L \mathcal D \mathcal L^\top, \qquad \mathcal{D} = \text{blkdiag}(A_0, -S_1, -S_2, \ldots, -S_N). \]
Since $\mathcal L$ has unit diagonal, we can define it as
\[
\mathcal{L}=\left[\begin{array}{ccccc}
    I & 0 & 0 & \cdots & 0 \\
	F_{1, 1}A_0^{-1} & I &  &  & \\
	F_{2, 1}A_0^{-1}&  F_{2, 2}S_1^{-1} & I &  & \\
    \vdots &  &  & \ddots & \\
	F_{N, 1}A_0^{-1} & F_{N, 2}S_1^{-1} & F_{N, 3}S_2^{-1} & \cdots & I
\end{array}\right]
\]
\begin{theorem}
\label{theo1}
The Schur complements $S_k$ satisfy the following relation:
\begin{equation}
    \label{Sk}
  S_k =A_k + B_kA_0^{-1}B_k^{\top} - 
  \underbrace{\sum_{i = 2}^{k} F_{k, i}S_{i-1}^{-1}F_{k, i}^{\top}}_{\text{\normalsize $\Delta_k $}} \equiv A_k + B_kA_0^{-1}B_k^{\top} - \Delta_k
\end{equation}
where $F_{k, 1} =B_k$ and
\begin{equation} \label{Fk}
  F_{k, m} = F_{k, 1} A_0^{-1} F_{m-1, 1}^{\top} - 
	\sum_{j=2}^{m-1} F_{k, j}S_{j-1}^{-1}F_{m-1, j}^{\top}. \end{equation}
\end{theorem}
\begin{proof}
  The matrix $\mathcal{D}\mathcal{L}^{\top}$ can be written as \[
	\mathcal{D}\mathcal{L}^{\top} = 
	\left[\begin{array}{ccccc}
	  A_0 & F_{1, 1}^{\top} &  F_{2, 1}^{\top} &  \cdots &  F_{N, 1}^{\top} \\ 
	  0 & -S_1 & -F_{2, 2}^{\top} &  \cdots &  -F_{N, 2}^{\top} \\ 
	  0 & 0 & -S_2 & & -F_{N, 3}^{\top} \\ 
	  0 & 0 & 0 & \ddots & \vdots \\
	  0 & 0 & 0 & 0 & -S_{N}
	\end{array}\right]
  \] 
  Let $r_i$ and  $c_j$ denote the  $i$-th row of  $\mathcal{L}$ and  $j$-th column of 
  $\mathcal{D}\mathcal{L}^{\top}$ respectively (starting from $0$).
We have $B_k = \mathcal{A}_{k, 0} =  r_{k} \cdot c_0 = F_{k, 1} $.
 For $k > m > 1$ we have  $0 = \mathcal{A}_{k, m-1} = l_{k} \cdot r_{m-1}$. This yields 
 \[
\left(
  F_{k, 1}A_1^{-1}, F_{k, 2}S_1^{-1}, \cdots, 
  F_{k, m}S_{m-1}^{-1}
  ,\cdots, F_{k, k}S_{k-1}^{-1}, I, 0
\right) \cdot \left(
  F_{m-1, 1}^{\top}, - F_{m-1, 2}^{\top}, \cdots, -F_{m-1, m-1}^{\top}, -S_{m-1}, 0
\right) = 0
\] which implies \eqref{Fk}. 
Finally, we see that $-A_k = \mathcal{A}_{k, k} = r_k \cdot c_k$. This yields
 \[\left(
  F_{k, 1}A_1^{-1}, F_{k, 2}S_1^{-1}, \cdots, 
   F_{k, k}S_{k-1}^{-1}, I, 0
\right) \cdot \left(
  F_{k, 1}^{\top}, - F_{k, 2}^{\top}, \cdots, -F_{k, k}^{\top}, -S_{k}, 0
\right) = -A_k \] which gives \eqref{Sk}.
\end{proof}
The recursive nature of the Schur complements just defined makes them unpractical in realistic applications, 
even for $N = 2$. We give below the explicit expressions for $S_1$, $S_2$ and $S_3$:
\begin{eqnarray}  S_1 & = & A_1 + B_1 A_0^{-1} B_1^\top \nonumber \\
\label{S2} S_2 &= &A_2 + B_2 A_0^{-1} B_2^\top  - (B_2 A_0^{-1} B_1^\top)(A_0 + B_1 A_0^{-1} B_1^\top)^{-1} (B_1 A_0^{-1} B_2^\top) \\
S_3 & = & A_3 + B_3 A_0^{-1} B_3^\top - B_3 A_0^{-1} B_1^\top S_1^{-1} B_1 A_0^{-1} B_3^\top - F_{22} S_2^{-1} F_{22}^\top,\nonumber \\
&& F_{22} = B_3 A_0^{-1} B_2^\top - B_3 A_0^{-1} B_1^\top S_1^{-1} B_1 A_0^{-1} B_2^\top. \nonumber
\end{eqnarray}
\noindent
We now state a results on the positive definiteness of the Schur complements, which generalizes \cite[Prop 1.]{BradleyGreif2026}
\begin{theorem}
\label{theo2}
    If $\mathcal A$ is invertible then all Schur complement are SPD.
\end{theorem}
\begin{proof}
By induction. For $N=0$ this is trivially satisfied since
$\mathcal A \equiv A_0$, which is SPD by assumption.
Let the statement holds for $j < N$, which means that
all $S_i, i \le j$ are SPD. We define the matrix
\begin{equation}
    \label{SP}
 \mathcal  A_j = \begin{bmatrix} A_0 & \mathbf B_j^\top \\ \mathbf B_j & -\mathbf A_j\end{bmatrix}, \ \mathbf B_j = \begin{bmatrix} B_1^\top & \ldots & B_j^\top\end{bmatrix}^\top, \ \mathbf A_j = \text{blkdiag}(A_1, \ldots, A_j).\end{equation}
We also define $\mathbf S_j$ the Schur complement related to matrix $\mathcal A_j$, namely
$\mathbf S_j = \mathbf A_j + \mathbf B_j A_0^{-1} \mathbf B_j^\top$. Since by inductive hypothesis
all block Schur complements $S_1, \ldots, S_j$ are SPD,  $\mathcal A_j$ is invertible and
hence $\mathbf S_j $ is SPD.
We now write $\mathcal A_{j+1}$ in terms of $\mathcal A_j$ as
\[ \mathcal A_{j+1}  = 
 \begin{bmatrix} A_0 & \mathbf B_j^\top & B_{j+1}^\top\\ \mathbf B_j & -\mathbf A_j & \bf 0\\
 B_{j+1} & \mathbf 0 & -A_{j+1}\end{bmatrix} = 
\begin{bmatrix} \text{\Large $\mathcal A_j$} & \begin{array}{l} B_{j+1}^\top\\\mathbf 0 \end{array} \\ \begin{array}{ll} B_{j+1} & \mathbf 0 \end{array}  & -A_{j+1} \end{bmatrix}. \]
Then, by definition of the Schur complement, and exploiting the inverse of a saddle-point matrix, 
\begin{eqnarray*}  S_{j+1} &= &  A_{j+1} + \begin{bmatrix} B_{j+1} & \mathbf 0 \end{bmatrix} \mathcal A_j^{-1}  \begin{bmatrix}  B_{j+1}^\top \\\mathbf 0\end{bmatrix}= A_{j+1} + B_{j+1} (\mathcal A_j^{-1})_{11} B_{j+1}^\top\\
&=& A_{j+1}+
 B_{j+1} \left(A_0^{-1} - A_0^{-1} \mathbf B_j^\top \mathbf S_j^{-1} \mathbf B_j A_0^{-1}\right)
 B_{j+1}^\top = \\
 &=& A_{j+1}+
 B_{j+1} A_0^{-1/2} \underbrace{\left(I - A_0^{-1/2} \mathbf B_j^\top \mathbf S_j^{-1} \mathbf B_j A_0^{-1/2}\right)}_{\text{\large{$Z$}}}A_0^{-1/2}B_{j+1}^\top.
 \end{eqnarray*}
$Z$ is a symmetric orthogonal projector. $Z = I - W W^\top$, where $W = A_0^{-1/2} \mathbf B_j^\top \mathbf S_j^{-1/2}$. $W$ is a rectangular (skinny) matrix. The smallest eigenvalue
of $Z$ is $1 - \rho(W W^\top) = 1 - \rho(W^\top W)$.
Since \[W^\top W=    S_j^{-1/2}  \mathbf B_j A_0^{-1}
 \mathbf B_j^\top \mathbf  S_j^{-1/2} =
  \mathbf S_j^{-1/2}  (\mathbf S_j + A_j)\mathbf  S_j^{-1/2} = I -  \mathbf S_j^{-1/2}  A_j \mathbf S_j^{-1/2}, \]
  we have that $\rho(W W^\top) \le 1$ and therefore $Z$ is positive semidefinite,
  and so is $S_{j+1}$.
 The definiteness follows from invertibility of $\mathcal A$, together with the Sylvester inertia law.
 \begin{Remark}
     By standard linear algebra, the expressions of the generic Schur complement $S_k$ in Theorems
     \ref{theo1} and \ref{theo2} can be shown to coincide.
 \end{Remark}

\end{proof}

% \cite{SZ}, where bounds on the condition number of the preconditioned matrix have been developed, 
%using the exact $\mathcal P_D$ as the preconditioner,
%the eigenvalues of the preconditioned matrix are exactly computed, in the simplified case where $A_k$ are $n_k \times n_k$ 
%zero matrices for $k > 0$, and are related to the zeros of suitable Chebyshev--like polynomials. 
%In the recent paper \cite{bergamaschi2025eigenvalue} these results are extended to the case where all diagonal blocks $A_k$
%can be different from zero.

  \section{Block diagonal preconditioners}
  \label{sec:BDP}
The process of
developing efficient (Schur complement based) block diagonal preconditioners starts from 
defining ideal preconditioners and then approximating
the Schur complements to make them
applicable in realistic settings. 
\subsection{Optimal Schur complement preconditioners: Forward derivation}
Which is the optimal preconditioner for an arrow-shaped multiple saddle-point linear system? To answer this question we rewrite the linear system as a regular
saddle-point system \eqref{regsp}
    \[
     \mathcal  A = \begin{bmatrix} A_0 & \mathbf B^\top \\ \mathbf B & -\mathbf A_+\end{bmatrix}
      =\begin{bmatrix} I  & 0 \\ \bB A_0^{-1}  & I\end{bmatrix}   \
     \begin{bmatrix} A_0  & 0 \\ 0 & -\mathbf S_+^{\text{dense}}\end{bmatrix}   \
     \begin{bmatrix} I & A_0^{-1} \bB^\top \\ 0 & I\end{bmatrix}.   \
    \]
 The optimal block diagonal and
  triangular preconditioners read as
  \[ \mathcal P_D =\begin{bmatrix} A_0 & 0 \\  0 &\mathbf S_+^{\text{dense}}\end{bmatrix},\qquad \mathcal P_T = \begin{bmatrix} A_0 & \mathbf B^\top \\  0 &-\mathbf S_+^{\text{dense}}\end{bmatrix}, \] 
 \[\text{where} \qquad \mathbf S_+^{\text{dense}} = \mathbf A_+ + \mathbf B A_0^{-1} \mathbf B^\top,
 \qquad \text{with} \quad
 (\mathbf S_+^{\text{dense}})_{kl} = \delta_{kl} A_k + B_k A_0^{-1} B_l^\top. 
% = \begin{bmatrix} B_1 A_0^{-1} B_1^\top & \ldots &  B %A_0^{-1} B^\top\\ \ldots & \ldots & \ldots \\
 % B_1 A_0^{-1} B^\top & \ldots &  B A_0^{-1} B^\top %\end{bmatrix}
  \]
  is a full block matrix. The eigenvalues of $\mathcal P_D^{-1} \mathcal A$ are, as well known,
  in $[-1, -0.618] \cup  [1, 1.618]$, while those
  of $\mathcal P_T^{-1} \mathcal A$ are all ones.

  \noindent
  A truly block diagonal preconditioner must instead have the following form
  \[ \mathbf S_+ = \text{blkdiag}(S_1, \ldots, S_N).\]
  The full Schur complements given by \eqref{Sk} cannot be computed in realistic problems. Moreover, we will show in Section \ref{sec:eig} that this choice is also detrimental in terms of eigenvalue
  distribution and hence of the number  of MINRES iterations.
  
  To obtain a practical preconditioner, we will approximate the Schur complements by
  \begin{equation} \label{incomplete}  \widetilde{\mathbf S}_+ =
  \text{blkdiag}( A_k + B_k  A_0^{-1} B_k^\top).
   \end{equation}
   Notice that this choice corresponds (a) to neglect $\Delta_k$ in \eqref{Sk},
   or (b) to define $\mathbf {\widetilde {S}_+} = \text{blkdiag}(\mathbf S_+^{\text{dense}})$.

   %Then each of the individual Schur complements in $\widetilde{\mathcal S}_+$ should be in their turn 
%a%pproximated by an inexact counterpart. The final preconditioner
%reads
%\[ \widehat {\mathcal{P}}_D = \begin{bmatrix} \widehat A_0 & 0 \\ 0 & \text{blkdiag}(\widehat A_k + {B_k \widehat %A_0^{-1} B_k^\top})\end{bmatrix}\]
\subsection{Optimal Schur complement preconditioners: Backward derivation}
Assume that all $A_k, k \ge 0$ are SPD.  Considering again the block subdivision as in \eqref{regsp},
we can define an optimal block diagonal preconditioner as 
\[ \mathcal P_D^{\text{rev}} =\begin{bmatrix}  A_0 +\mathbf B^\top \mathbf  A_+^{-1} \mathbf  B &0 \\ 0 & \mathbf A_+  \end{bmatrix}. \]
We will denote this as \textit{reverse} Schur complement preconditioner as
named in \cite{FRANCESCHINI2022114632}.
The eigenvalues of the preconditioned matrix are, again, 
in $[-1, -0.618] \cup [1, 1.618]$ providing fast
MINRES convergence. 
The optimal triangular preconditioner is defined as 
\[ \mathcal P_T^{\text{rev}} =\begin{bmatrix}  A_0 +\mathbf B^\top \mathbf  A_+^{-1} \mathbf  B & \mathbf B^\top \\ 0& -\mathbf A_+    \end{bmatrix}, \]
yielding a preconditioned GMRES \cite{saad1986gmres} iteration converging in exactly 2 iterations.

These preconditioners
rest on the practical computation and inversion of the Schur complement
\[ \Sigma =  A_0 +\mathbf B^\top \mathbf  A_+^{-1} \mathbf  B = A_0 + \sum_{k=1}^N B_k^\top A_k^{-1} B_k. \]
Irrespective of the invertibility of the diagonal blocks, the following approximate block-diagonal preconditioner can be defined as
\[ \widehat {\mathcal P}_D^{\text{rev}} =
\begin{bmatrix}  
\widehat{A}_0 +\mathbf B^\top \widehat {\mathbf A}_+^{-1} \mathbf B &0 \\ 0 & \widehat{\mathbf A}_+  
\end{bmatrix}, \]
with $\widehat A_k$ SPD approximations to  $A_k$.

\section{Eigenvalues of Preconditioned Saddle-Point Systems}\label{sec:eig}

\subsection{Forward Schur complement preconditioner}
The preconditioned matrix reads 
\begin{equation}\label{Q}
    \begin{aligned}
    \mathcal{Q}&={\mathcal{P}}_D^{-1/2} \mathcal{A}{\mathcal{P}}_D^{-1/2}=\left[\begin{array}{ccccc}
        E_0 & \mathbf R^\top\\
        \mathbf R & -\mathbf E_+\\
    \end{array}\right],
    \qquad \mathbf R = \begin{bmatrix} R_1 \\ R_2 \\ \ldots \\ R_N\end{bmatrix}, \qquad
   \mathbf E_+ = \text{blkdiag}(E_1, \ldots, E_N) \\
    %&=\left[\begin{array}{ccccc}
     %   E_0 & R_1^T & 0 & & \\
      %% 0 & R_2 & \overline{S_2}-R_2R_2^T & \hspace{-12mm}\ddots \\
        % & \ddots & \ddots & R^T\\
        % & R&  R&\hspace{-8mm} (-1)^N(\overline{S_N}-RR^T)\\
    %\end{array}\right]
    \end{aligned}
\end{equation}
where
\[ R_k = \widehat{S}_{k}^{-1/2}B_k\widehat{A}_0^{-1/2}, \quad
   E_k = \widehat{S}_k^{-1/2}A_k\widehat{S}_k^{-1/2}, \ \ k = 0, \dots, N.
\]

Componentwise, we write the eigenvalue problem $\mathcal{Q}u=\lambda u$, with $u=\left[u_1^\top,u_{+}^\top\right]^\top$, as 

    \begin{eqnarray}
  \nonumber      (E_0-\lambda I)u_1  +  \mathbf R^Tu_+ &= & 0\\
    \mathbf R u_1 - (\lambda I + \mathbf E_+) u_{+} &= & 0
\label{eq1}
    \end{eqnarray}

The matrices $R_kR_k^T$ are all symmetric and positive definite. We define two indicators $\gamma_E^{(k)}$ and $\gamma_R^{(k)}$ using the Rayleigh quotient
	\begin{equation}
		\begin{array}{lll}
\alpha_E^{(i)} \equiv \lambda_{\min }\left(E_i\right), & \beta_E^{(i)} \equiv \lambda_{\max }\left(E_i\right), & \gamma_E^{(i)}\left(w\right)=\dfrac{w^{T} E_i w}{w^{T} w} \in\left[\alpha_{E}^{(i)}, \beta_{E}^{(i)}\right], \quad i=0,\dots,N \\[.6em]
\alpha_R^{(i)} \equiv \lambda_{\min }\left(R_i R_i^{T}\right), & \beta_R^{(i)} \equiv \lambda_{\max }\left(R_i R_i^{T}\right), & \gamma_R^{(i)}\left(w\right)=\dfrac{w^{T} R_i R_i^{T} w}{w^{T} w} \in\left[\alpha_{R}^{(i)}, \beta_{R}^{(i)}\right], \quad i=1,\dots,N.
\end{array}
		\label{gammaER}
	\end{equation}

Moreover we have 
\begin{eqnarray} 
\mathbf R^\top \mathbf R &=& \sum_{k=1}^N R_k   ^\top R_k, \qquad  \mathbf R^\top \mathbf R \in \mathbb  R^{n_0 \times n_0} \\
(\mathbf R \mathbf R^\top)_{kl} &=& R_k R_l^\top, \qquad \mathbf R \mathbf R^\top
\in \mathbb R^{n_1 + \ldots + n_N \times n_1 + \ldots + n_N}
\end{eqnarray}
We define as $\Gamma_R(w)$ and  $\Gamma_E(w)$ the Rayleigh quotients
of $\bR \bR^\top$ and
$\mathbf E_+$, respectively %matrices since it turns out that
\begin{eqnarray}
0 < {\boldsymbol \alpha}_R \le  \lambda_{\min}(\mathbf R \mathbf R^\top)
\le \Gamma_R(w)
&\le& \lambda_{\max}(\mathbf R \mathbf R^\top)
\le {\boldsymbol \beta}_R \equiv \sum_{k=1}^N \beta_R^{(k)}. \nonumber \\[-.8em] 
0\le \lambda_{\min}(\mathbf R^\top \mathbf R)
\le \sum_{k=1}^N \gamma_R^{(k)} 
&\le &\lambda_{\max}(\mathbf R^\top \mathbf R)
\le {\boldsymbol \beta}_R \\[.3em]
0 \le {\boldsymbol \alpha}_E = \min_{k\ge 1}\{\alpha_E^{(k)} \}\le  \Gamma_E(w) &\le& \max_{k \ge 1}\{\beta_E^{(k)}\} = \boldsymbol \beta_E. \nonumber
\end{eqnarray}
\begin{Remark}
\label{rem1} 
If the Schur complements in the preconditioner are as in \eqref{incomplete} and they are formed exactly,
    then the matrices $R_k$ and $E_k$ are related as described below:
    \[ R_k R_k^\top = \widehat{S}_{k}^{-1} B_k  A_0^{-1} B_k^\top \widehat S_k^{-1/2}=
    \widehat{S}_k^{-1/2}(\widehat S_k - A_k) \widehat{S}_k^{-1/2}
    =I - E_k. \]
  %
%Particular cases:
%    \begin{enumerate}
 %   \item If $k=1$ then $\Delta_1 = 0$, and $ R_k R_k^\top = I - E_k$. It follows that $\beta_R^{(1)} \le 1.$
%\item Let $k \ge 2$.    If $A_k = 0$ then $\gamma_R > 1$, which also implies that
 %   $\gamma_R + \gamma_E > 1$. 
  %% $\widehat S_k = \widehat A_k + B_k \widehat A_0^{-1} B_k^\top$
   % then  $R_k R_k^\top = I - E_k$, and $\gamma_R \le 1$.
   % \end{enumerate}
   % If the reverse Schur complement approach si considered then, for $k > 0$,
%\[ R_k R_k^\top = A_{k}^{-1} B_k  A_0^{-1} B_k^\top A_k^{-1/2}=
 %   \widehat{S}_k^{-1/2}(\widehat S_k - A_k) \widehat{S}_k^{-1/2}
 %5   =I - E_k. \]
\end{Remark}
\noindent
\subsection{Reverse Schur complement preconditioner}
The preconditioned matrix is $\mathcal Q$ as in \eqref{Q} with the following block definitions:
\begin{eqnarray*}
E_0 = \widehat{S}_0^{-1/2} A_0 \widehat{S}_0^{-1/2}, && \widehat{S}_0 = \widehat{A}_0 +{\bB}^\top \widehat{\mathbf{A}}_+^{-1} \mathbf  B, \quad E_k = \widehat {A}_k^{-1/2} A_k \widehat {A}_k^{-1/2}\qquad R_k = \widehat{A}_k^{-1/2} B_k\widehat{S}_0^{-1/2}, \quad \\
\end{eqnarray*}
We also have
\[ \bR^\top \bR = \sum_k R_k^\top R_k = \widehat S_0^{-1/2} \mathbf B^\top \mathbf  A_+^{-1} \mathbf  B S_0^{-1/2} =\widehat S_0^{-1/2} (\widehat S_0 - \widehat A_0)  \widehat S_0^{-1/2} =
I - \widehat S_0^{-1/2} \widehat A_0\widehat S_0^{-1/2}= I - E_0,\]
implying that
\[ 0 \le \Gamma_R \le 1-\alpha_E^{(0)} .\]

%\noindent
%\section{Eigenvalue analysis}
The analysis that follows refers to both the forward and the reverse Schur complement preconditioner, since the matrix $\mathcal Q$ looks like  the same for both approaches, while what changes is the definition of the single blocks.
We start mentioning a technical Lemma, which generalizes \cite[Lemma 1]{BergaNLAA10} whose proof can be found in \cite{Bergamaschi-COAP-2024}.
\begin{lemma}
\label{Lemma1}
    \textit{Let Y be a symmetric matrix valued function defined in $F\subset\mathbb{R}$ and}
    \[
        0\notin[\min\{\sigma(Y(\zeta))\},\max\{\sigma(Y(\zeta))\}]\quad \quad \textit{for all }\ \zeta\in F.
    \]
    \textit{Then, for arbitrary $s\neq0$, there exists a vector $v\neq 0$ such that}
    \[
       \frac{s^TY(\zeta)^{-1}s}{s^Ts}= \frac{1}{\gamma_Z} \quad\quad \textit{with }\ \gamma_Z = \frac{v^TY(\zeta)v}{v^Tv}.
    \]
\end{lemma}
\noindent

\noindent
\begin{theorem}
\label{theoremUpsilon}
Let ${\mathcal P}_D$ a Schur complement block diagonal preconditioner. 
Then the eigenvalues of 
${\mathcal P}_D^{-1} \mathcal A$ are located in
$[\lambda_-^{LB}, \lambda_-^{UB}]\ \cup
[\lambda_+^{LB}, \lambda_+^{UB}]
$, where
\begin{eqnarray}
    \lambda_-^{LB} & = &\frac{\alpha_E^{(0)} - \boldsymbol \beta_E
- \sqrt{(\alpha_E^{(0)} + \boldsymbol \beta_E)^2 + 4  \boldsymbol \beta_R}}{2}\\
    \lambda_-^{UB} & =& \frac{\beta_E^{(0)} - \boldsymbol \alpha_E
- \sqrt{(\beta_E^{(0)} + \boldsymbol \alpha_E)^2 + 4  \boldsymbol \alpha_R}}{2}\\
    \lambda_+^{LB} & = & \alpha_E^{(0)} \\
    \lambda_+^{UB} & = &
\frac{\beta_E^{(0)} - \boldsymbol \alpha_E
+ \sqrt{(\beta_E^{(0)} + \boldsymbol \alpha_E)^2 + 4  \boldsymbol \beta_R}}{2} 
\end{eqnarray}
\end{theorem}
\begin{proof}
	We consider the eigenvalue problem \eqref{eq1}.
If $\lambda \not \in [\alpha_E^{(0)}, \beta_E^{(0)}]$ then
\[ u_1 = (\lambda I - E_0)^{-1} \mathbf R^\top u_+,\]
and
\[ \left(\mathbf R (\lambda I - E_0)^{-1} \mathbf R^\top - \lambda \mathbf I - \mathbf E_+\right)u_+ = 0.\]
Premultiplying by $-\dfrac{u_+^T}{u_+^\top u_+}$ and applying Lemma \ref{Lemma1}, yield, after defining $s = \mathbf R^\top u_+$:	 
\begin{equation} 
0 = \lambda +\frac{u_+^\top \mathbf E_+ u_+}{u_+^\top u_+} - \frac{1}{\lambda - \dfrac{v^\top E_0 v}{v^\top v}} \frac{u_+^\top \mathbf R \mathbf R^\top u_+}{u_+^\top u_+} = \lambda + \Gamma_E -\frac{\Gamma_R}{\lambda - \gamma_E^{(0)}}, \qquad \text{or, equivalently,}
\end{equation}
\[ \lambda^2 + (\Gamma_E - \gamma_E^{(0)})\lambda - \gamma_E^{(0)}\Gamma_E - \Gamma_R=0,\]
and therefore any eigenvalue of $\mathcal P_D^{-1} \mathcal A$ not lying in $[\alpha_E^{(0)}, \beta_E^{(0)}]$ can be described as
\[ \lambda_{\pm}
= 
\frac{\gamma_E^{(0)} - \boldsymbol \gamma_E
\pm \sqrt{(\gamma_E^{(0)} + \boldsymbol \gamma_E)^2 + 4  \boldsymbol \gamma_R}}{2}.
\]
Observing that 
\begin{itemize} 
\item $\lambda_-$ is decreasing in $\boldsymbol \gamma_E, \boldsymbol \gamma_R$ increasing in $\gamma_E^{(0)}$ 
\item $\lambda_+$ is increasing in $\gamma_E^{(0)}, \boldsymbol \gamma_R$, decreasing in $\boldsymbol \gamma_E$ 
\end{itemize}
and that $\lambda_+ \ge \gamma_E^{(0)} \ge \alpha_E^{(0)}$, the thesis follows
\end{proof}

\noindent
The previous spectral analysis is relatively simple. In particular situations, namely for large $N$, it can however lead to loose bounds since they depend on 
the global indicators $\boldsymbol \alpha_E, \boldsymbol \beta_E, \boldsymbol \alpha_R, \boldsymbol \beta_R$. Some bounds can be refined, by considering instead the single 
indicators $ \alpha_E^{(j)},  \beta_E^{(j)},  \alpha_R^{(j)},  \beta_R^{(j)}$, as stated in the following result.

%Let now $\mathcal I_E^- = \bigcup_{k=1}^{N} I_{E_{k}}$, then
\noindent
We define the function $\Pi: (-\infty, - \boldsymbol \beta_E) \cup (-\boldsymbol \alpha_E, + \infty) \to \mathbb R, \qquad  
\displaystyle\Pi(\lambda) \equiv \lambda - \gamma_E^{(0)} - \sum_{k=1}^{N}\frac{\gamma_R^{(k)}}{\lambda +\gamma_E^{(k)}}.$
\begin{theorem}\label{theoremPI}
    The eigenvalues of $\mathcal{P}_D^{-1}\mathcal{A}$ are located in $\displaystyle  [\mu_-^{LB}, -\boldsymbol \alpha_E] 
    \cup [\alpha_E^{(0)}, \mu_+^{UB}] $, where
    \begin{itemize}
    \item $\mu_-^{LB}$ is the negative zero of $\Pi$ with parameters: $\gamma_E^{(0)}=\alpha_E^{(0)}; \gamma_R^{(j)} =\beta_R^{(j)}, \gamma_E^{(j)} = \beta_E^{(j)}, j>0$. 
% \item $\mu_+^{LB}$ is the positive zero of $\Pi$ with parameters: $\gamma_E^{(0)}=\alpha_E^{(0)}; \gamma_R^{(j)}=0, \gamma_E^{(j)} =\beta_E^{(j)}, j>0$. 
    \item $\mu_+^{UB}$ is the positive zero of $\Pi$ with parameters: $\gamma_E^{(0)}=\beta_E^{(0)}; \gamma_R^{(j)}=\beta_R^{(j)}, \gamma_E^{(j)} =\alpha_E^{(j)}, j>0$. 
    \end{itemize}
\end{theorem}

\begin{proof}
Let $\lambda \not \in [-\boldsymbol \beta_E, -\boldsymbol \alpha_E]$, then $E_j + \lambda I$
	is definite and hence invertible. Considering \eqref{eq1}, we can write
\[ u_+ = (\lambda \mathbf I + \mathbf E_+)^{-1} \mathbf R^\top u_1. \]
Upon substitution, the first equation reads
\[(  E_0 - \lambda I + \mathbf R^\top(\lambda \mathbf I + \mathbf E_+)^{-1} \mathbf R)u_1  = 0 .\]
Premultiplying by $-\dfrac{u_1^T}{u_1^T u_1}$, defining $s_k = R_k u_1$, and applying Lemma \ref{Lemma1} to $\dfrac{s_k^\top (\lambda I + E_k)^{-1} s_k}{s_k^\top s_k}$, yield:	 
\begin{equation} \label{Pidef} 0 =  -\frac{u_1^\top E_0 u_1}{u_1^\top u_1} + \lambda  - \sum_{k=1}^N \frac{u_1 R_k^\top R_k u_1}{u_1^\top u_1} \frac{1}{\lambda + \dfrac{v_k ^\top E_k v_k}{v_k^\top v_k}} 
= -\gamma_E^{(0)} + \lambda -\sum_{k=1}^{N}\frac{\gamma_R^{(k)}}{\lambda +\gamma_E^{(k)}} \equiv \Pi(\lambda),\end{equation}
with $\Pi: (-\infty, -\boldsymbol \beta_E) \cup (-\boldsymbol \alpha_E, + \infty) \to \mathbb R$.
Since $\displaystyle
\Pi'(\lambda) = 1 +
 \sum_{k=1}^{N}\frac{\gamma_R^{(k)}}{(\lambda +\gamma_E^{(k)})^2}>0$, the function $\Pi(\lambda)$ is increasing in every interval in which it is defined.
 \[ \text{Since also} \quad \lim_{\lambda \to -\infty} \Pi(\lambda) = -\infty\ \
  \lim_{\lambda \to -{\boldsymbol \beta}_E^-} \Pi(\lambda) = +\infty\qquad
   \lim_{\lambda \to -{\boldsymbol \alpha}_E^+} \Pi(\lambda) = -\infty\ \
  \lim_{\lambda \to +\infty} \Pi(\lambda) = +\infty,\]
$\Pi$ has exactly two zeros, separated by the endpoints of the intervals of its definition. 
If $\boldsymbol \alpha_E> 0$, then $-\boldsymbol \alpha_E$ gives a bound for 
the largest negative eigenvalue. In case $\boldsymbol \alpha_E = 0$ the bound is $0$ and hence is not descriptive.
%\textcolor{red} {Since $\Pi(\gamma_E^{(0)}) < 0$, the present lower bound for the positive eigenvalues will improve that of Theorem \ref{thm:bound1}}.
%\subsection{How extremal zeros move with parameters.}
Easy computation shows
\begin{eqnarray}
    \frac{\partial \Pi}{\partial \gamma_E^{(j)}} &=& \begin{cases}
    -1 & \text{if} \ j =0\\  \dfrac{\gamma_R^{(j)}}{(\lambda + \gamma_E^{(j)})^2} & \text{if} \ j >0 \end{cases}
    \quad \longrightarrow \ \textnormal{sgn}\left( \frac{\partial \Pi}{\partial \gamma_E^{(j)}}\right) = 
    \begin{cases}
    -1 & \text{if} \ j =0\\  +1 & \text{if} \ j >0 \end{cases}
    \\
    \frac{\partial \Pi}{\partial \gamma_R^{(j)}} &=&
    -\frac{1}{\lambda + \gamma_E^{(j)}}.
\end{eqnarray}
By the implicit function theorem \cite[Eq. (17)]{pilotto-et-al2026}, since the function $\Pi(\lambda)$ is locally increasing, the monotonicity of its zeros as a function of the parameters reverts that of the polynomial.
The following table summarizes the monotonicities of
the three bounding roots, with respect to the parameters.

\begin{tabular}{lccc||ccc}
& $\dfrac{\partial \lambda}{\partial \gamma_R^{(j)}}(\xi) $
& $\dfrac{\partial \lambda}{\partial \gamma_E^{(0)}}(\xi) $
& $\dfrac{\partial \lambda}{\partial \gamma_E^{(j)}}(\xi)$ & $\gamma_R^{(j)}$ & $\gamma_E^{(0)} $ &
$\gamma_E^{(j)}, j > 0 $ \\[1.2em] \hline
 Smallest negative root $\xi$&  $-$ &  $+$ & $-$ & $\beta_R^{(j)}$&$\alpha_E^{(0)}$ &  $\beta_E^{(j)}$\\
 %Largest negative root $\xi$&  ${\color{blue} - }$ &  $+$ & $-$&
 %${\color{blue}\alpha_R^*}, \beta_E^{(0)}; \  %\alpha_E^{(j)}, j > 0$\\
Smallest positive root $\xi$&  $+$ &  $+$ & $-$&0&$\alpha_E^{(0)}$ &  $\beta_E^{(j)}$\\
 Largest positive root $\xi$&  $+$ &  $+$ & $-$&
 $\beta_R^{(j)}$&$\beta_E^{(0)}$&$\alpha_E^{(j)} $
 \\ 
\end{tabular}

\noindent
The thesis follows, by observing that $\Pi(\xi)$ for
$\gamma_R^{(j)} = 0, \gamma_E^{(0)} = \alpha_E^{(0)}$ simplifies to $\xi - \alpha_E^{(0)}$.
\end{proof}
\begin{Remark}
Let $N=2$. If we write  \[\Pi(\lambda) = \frac{\Pi_{\text{num}}(\lambda)}
{\Pi_{\text{den}}(\lambda)},\] with the parameters corresponding to
the smallest negative root, 
then $\Pi_{\text{num}}(\lambda)$ is the same as \cite[(23a)]{BradleyGreif2026}. If the parameters are those for the largest positive root, we get (23b) in the same paper.
\end{Remark}
\noindent
The results in Theorems \ref{theoremUpsilon} and \ref{theoremPI} are combined in the 
following Corollary.
\begin{corollary}
\label{corollary}
    The eigenvalues of the preconditioned multiple saddle-point matrix are in the following intervals:
\[[\mu_-^{LB}, \lambda_-^{LB}] 
    \cup [\alpha_E^{(0)}, \mu_+^{UB}].\]
\end{corollary}
\begin{proof}
    To show that  $\lambda_-^{LB} < -\boldsymbol \alpha_E$, we evaluate $\Upsilon$ with parameters $\boldsymbol \alpha_R, \beta_E^{(0)}, \boldsymbol \alpha_E$ on $\lambda  = -\boldsymbol \alpha_E$ to get $\Upsilon(-\boldsymbol \alpha_E) = -\Gamma_R < 0$. 
    
    Second, denoting $\xi \equiv \lambda_+^{UB}$  it satisfies $\xi - \beta_E^{(0)} = \dfrac{\boldsymbol \beta_R}{\xi +\boldsymbol \alpha_E}.$ Now we evaluate on $\xi$ the function $\Pi$ with the parameters
    used for computing the upper bound for the positive eigenvalues. With this choice we have
     $\displaystyle \Pi(\xi) = \dfrac{\boldsymbol \beta_R}{\xi + \boldsymbol \alpha_E}- \sum_{k=1}^{N}\frac{\beta_R^{(k)}}{\xi +\alpha_E^{(k)}}.$ Let $j$ such that $\boldsymbol \alpha_E = \min_k \alpha_E^{(k)} = \alpha_E^{(j)}$, then
     \[ \Pi(\xi) = 
     \frac{\boldsymbol\beta_R}{\xi + \alpha_E^{(j)}}- \sum_{k=1}^{N}\frac{\beta_R^{(k)}}{\xi +\alpha_E^{(k)}} \ge
     \frac{\boldsymbol\beta_R}{\xi + \alpha_E^{(j)}}- \sum_{k=1}^{N}\frac{\beta_R^{(k)}}{\xi +\alpha_E^{(j)}} = \frac{\beta_R - \sum_{k=1}^N \beta_R^{(k)}}{\xi + \alpha_E^{(j)}} \ge 0,\]
     showing that $\lambda_+^{UB}  \ge \mu_+^{UB}.$ 
     
     Analogously, setting $\xi = \lambda_-^{LB}$, and evaluating
     the function $\Pi$, with the parameters used for computing the smallest negative eigenvalue, on $\xi$, yields,
      \[ \Pi(\xi) = 
     \frac{-\boldsymbol\beta_R}{-\xi -\beta_E^{(j)}}+ \sum_{k=1}^{N}\frac{\beta_R^{(k)}}{-\xi -\beta_E^{(k)}} \le
     \frac{-\boldsymbol\beta_R}{-\xi -\beta_E^{(j)}}+\sum_{k=1}^{N}\frac{\beta_R^{(k)}}{-\xi -\beta_E^{(j)}} = \frac{-\beta_R+\sum_{k=1}^N \beta_R^{(k)}}{-\xi- \beta_E^{(j)}} \le 0,\]
     where $j$ is such that $\boldsymbol \beta_E = \max_k \beta_E^{(k)} = \beta_E^{(j)}$. This shows that 
$\lambda_-^{LB}  \le \mu_-^{LB}.$
     \end{proof}

   \subsection{Validation of the bounds on synthetic test cases}
   To check the validity of the bounds we constructed a test case with a $5 \times 5$  arrow matrix with blocks
   of decreasing sizes: [1000 400 200 120 80]. The diagonal blocks $A_k, k = 0, \ldots, 4$ are diagonal matrix with uniformly distributed elements: $|(A_k)_{ii}| \in [a_k^{\min}, a_k^{\max}]$. $B_k, k = 1, \ldots, 4$ are sparse random matrices obtained with the Matlab command \texttt{$B_k$ = sprand($n_k, n_1, 0.1$)} and normalized so to have all norm equal to 10. The values of $a_k^{\min}, a_k^{\max}$ are reported in Table \ref{lambda}. 
   \begin{table}[h!]
	   \begin{center}
       \begin{tabular}{cccccc}
       $k$  & 0& 1 & 2 & 3 & 4 \\ \hline
       $a_k^{\min}$& 0.1 &  0.5 &  0.3 &  0.8&  0.01\\
             $a_k^{\max}$  &   2.1 &5.5 &1.3 &1.8 &1.25 \\
       \end{tabular}
	   \end{center}
\caption{Extremal eigenvalues of the matrices $A_k, k = 0, \ldots, 4$.}
   \label{lambda}
   \end{table}

   \noindent
   In the first experiment we choose as the preconditioner 
   \[ \mathcal P_D = \text{blkdiag($A_0$, blkdiag($A_k + B_k A_0^{-1} B_k^\top$))}.\]
   We compute the indicators, which are reported below
   \begin{center}
       
   \begin{minipage}{6cm}     
\begin{tabular}{c|rrrr}
  $k$    & $\alpha_R^{(k)} $ &  $\beta_R^{(k)} $ &  $\alpha_E^{(k)} $ &  $\beta_E^{(k)} $ \\
\hline
   1  & 0.076   &  0.972  &   0.028   &  0.924 \\
   2  & 0.138   &   0.950 &     0.050 &    0.862 \\
   3  & 0.171   &   0.893 &      0.107&     0.829 \\
   4 &  0.262   & 0.991   &  0.009    & 0.738 \\
\end{tabular}
   \end{minipage}
\hspace{2cm}
\begin{minipage}{6cm}
    Moreover we have \\
	$\boldsymbol \alpha_R = 7.65 \times 10^{-3}$\\[.6em]
    $\alpha_E^{(0)} = \beta_E^{(0)} = 1$.\\
    $\boldsymbol \alpha_E = \min\{\alpha_E^{(k)}, k > 0 \}$ = 0.009.
\end{minipage}

   \end{center}
\medskip
  
\noindent
  Applying Corollary \ref{corollary}, we found the following true eigenvalues and bounds:
  \smallskip

\begin{center}
\begin{tabular}{l|cccc}
	\hline
True eigenvalues  &     -1.4902  &    -0.0621     &  1    &   2.4206\\
Bounds &      -1.5441  &  -0.0165  &          1  &     2.4968\\
\end{tabular}
\end{center}
The MINRES solver with preconditioner $\mathcal P_D$ converged in 65 iterations after which the relative 
residual norm was below $10^{-8}$.

\noindent
The second test case considers the matrix above, while the preconditioner is 
$\mathcal P_D = \text{blkdiag($I$, blkdiag($A_k + B_k B_k^\top$))}$, with the following values of the indicators:
\begin{center}
\begin{minipage}{6cm}
\begin{tabular}{c|rrrr}
  $k$    & $\alpha_R^{(k)} $ &  $\beta_R^{(k)} $ &  $\alpha_E^{(k)} $ &  $\beta_E^{(k)} $ \\
\hline

  1  &  0.062  &    0.952 & 0.048   &     0.938 \\
  2  &  0.135  &    0.932 & 0.068   &     0.865   \\
  3  &  0.166  &    0.868 & 0.132   &     0.834   \\
  4  &  0.197  &    0.989 & 0.011   &     0.803     \\
\end{tabular}
   \end{minipage}
\hspace{2cm}
\begin{minipage}{6cm}
    Moreover we have \\
	$\boldsymbol \alpha_R = 8.76 \times 10^{-3}$\\[.6em]
    $\alpha_E^{(0)} = 0.1, \ \beta_E^{(0)} = 2.1$.\\
    $\boldsymbol \alpha_E = \min\{\alpha_E^{(k)}, k > 0 \}$ = 0.011.
\end{minipage}
\end{center}

\medskip

\noindent
  Again, invoking Corollary \ref{corollary}, we found the following true eigenvalues and bounds:
  \smallskip

\begin{center}
\begin{tabular}{l|cccc}
	\hline
True eigenvalues  &     -1.4992  &    -0.0946     &  0.2513    &   2.6328\\
Bounds &      -1.3737  &  -0.0153  &      0.1000  &     3.2344\\
\end{tabular}
\end{center}
For this case, in which the preconditioner Schur complements are applied inexactly (with    $A_0$ approximated by the identity matrix), the MINRES solver met
the requirement on the relative residual after 101 iterations.

All in all, the bounds reveal very close to the computed extremal eigenvalues, the upper bound for the negative eigenvalues being the loosest of the four.

\section{Application of the analysis to double saddle-point linear systems}
\label{sec:double}

\subsection{Double saddle-point linear system with zero diagonal blocks}
We consider a simplified setting, in which $N=2$, $A_1 = A_2 = 0$. 
 In this case,
\eqref{eq1} simplifies to
\begin{eqnarray}
E_0 u_1 + R_1 ^\top u_2 + R_2^\top u_3 &=& \lambda u_1\nonumber \\
R_1 u_1 & = & \lambda u_2\label{zeroblocks}\\ 
R_2 u_1 & = & \lambda u_3  \nonumber
\end{eqnarray}
Following the proof of Theorem \ref{theo2}, if $\lambda \not \in [\alpha_E^{(0)}, \beta_E^{(0}]$ then it satisfies $\Pi(\lambda) = 0$, that is
%If $R_1 u_1 = R_2 u_1 = 0$, then, premultiplying the first equation by $u_1^\top$ yields $\lambda = 1$. In the most
%general case, we can write the function $\Pi(\lambda)$ which reads 
\begin{equation}
    \label{eqdouble}
 \lambda - \gamma_E^{(0)}  -  \frac{\gamma_R^{(1)}}{\lambda} - \frac{\gamma_R^{(2)}}{\lambda}=0. \end{equation}
Application of Corollary 1, together with a straightforward analysis of the 2nd order equation derived by the previous, provides $\lambda \in [\mu_-^{\text{LB}},\lambda_-^{UB}) \cup [\alpha_E^{(0)},\mu_+^{\text{UB}}]$, \ where \
\begin{equation} \label{bounds} \mu_-^{\text{LB}} = \frac{\alpha_E^{(0)} - \zeta_-}{2}, \quad \lambda_-^{UB} = -\frac{2 \boldsymbol \alpha_R}
{\beta_E^{(0)} + \sqrt{\beta_E^{(0)} + 4 \boldsymbol \alpha_R}}, \quad 
\mu_+^{\text{UB}} = \frac{\beta_E^{(0)} + \zeta_+}{2}, \ \text{and}\end{equation}
\[  \zeta_- = \sqrt{(\alpha_E^{(0)})^2+ 4 \beta_R^{(1)} +4\beta_R^{(2)}},
\  \zeta_+ = \sqrt{(\beta_E^{(0)})^2+ 4 \beta_R^{(1)} +4\beta_R^{(2)}}.
 \] 
\noindent
\subsubsection{Forward incomplete preconditioner} Considering the  exact (incomplete) preconditioner \[\mathcal P_D =   \text{blkdiag}(A_0, B_1 A_0^{-1}B_1^\top, B_2 A_0^{-1}B_2^\top),\]
we have that $E_0 = R_1R_1^\top = R_2 R_2^\top = I$, all the indicators take the value one and therefore
\[ \lambda \in (-1,\lambda_-^{UB}) \cup [1,2],\]
with $\lambda_-^{UB} = -\dfrac{2 \boldsymbol \alpha_R}
{1 + \sqrt{1 + 4 \boldsymbol \alpha_R}}$. This result agrees with the bound in \cite[Theorem 4.1]{Benzi2018}.
%the one in \cite[Theorem 8]{BradleyGreif2026}. If $\boldsymbol \alpha_R$ is small then a simpler bounds is 
%$\lambda_-^{UB} = -\dfrac{\boldsymbol \alpha_R}{1+\boldsymbol \alpha_R}$.

\subsubsection{Forward full preconditioner} $\mathcal P_D$ is as before, except in the definition of $S_2$ which is now, following \eqref{Sk},
\[S_2 = B_2 A_0^{-1}B_2^\top - B_2 A_0^{-1} B^\top_1 (B_1 A_0^{-1} B_1^\top)^{-1} B_1 A_0^{-1} B_2^\top. \]
As before, $R_1 R_1^\top = I$. Regarding $R_2 R_2^\top = S_2^{-1/2} B_2 A_0^{-1} B_2^\top S_2^{-1/2}$, we can characterize its extremal eigenvalues considering
\[ X = (B_2 A_0^{-1} B_2^\top)^{-1} S_2 = I - (B_2 A_0^{-1} B_2^\top)^{-1}B_2 A_0^{-1} B^\top_1 (B_1 A_0^{-1} B_1^\top)^{-1} B_1 A_0^{-1} B_2^\top \equiv I - X_0,\]
which is similar to the inverse of $R_2 R_2^\top$. By the Cauchy-Schwartz inequality, and by Assumption \ref{ass2}, $\lambda_{\max} (X_0) < 1$ therefore $\sigma(X) \in (0,1]$ and  $\alpha_R^{(0)} = 1$, with $\beta_R^{(2)}$ potentially large. The bounds in \eqref{bounds} read in this case:
\[\mu_-^{\text{LB}} = \frac{1 - \sqrt{5 + 4 \beta_R^{(2)}}}{2}, \quad \lambda_-^{UB} = -\frac{2 \boldsymbol \alpha_R}
{1 + \sqrt{1 + 4 \boldsymbol \alpha_R}}, \quad 
\mu_+^{\text{UB}} = \frac{1 + \sqrt{5 +4\beta_R^{(2)}}}{2}.
\]
Note that this analysis completes the eigenvalue characterization in \cite[Theorem 8]{BradleyGreif2026}, which is limited to counting the multiplicity of a few selected eigenvalues. 
\smallskip

\noindent
\subsubsection{Reverse preconditioner}
\label{reverse}
Taking now into account the reverse Schur complement preconditioner we could set $\widehat A_1 = \delta I, \widehat A_2 = \delta I$, with $\delta$ small, since $A_1 = A_2 = 0$. Hence \[\mathcal P_D = \text{blkdiag}(\widehat S_0, \widehat A_1, \widehat A_2) =\text{blkdiag} (A_0 + \delta^{-1} B_1^\top B_1 + \delta^{-1}{B_2^\top B_2}, \delta I_{n_1}, \delta I_{n_2}).\]
We have $\bR^\top \bR = I - E_0$, therefore $\gamma_R^{(1)} + \gamma_R^{(2)} = 1 - \gamma_E^{(0)} $, and \eqref{eqdouble} becomes \[ 0 = \lambda^2 - \lambda \gamma_E^{(0)} - 1 + \gamma_E^{(0)} = (\lambda-1)(\lambda +1 - \gamma_E^{(0)}), \]
so that 1 is the only positive eigenvalues, while all negative ones take the form $\lambda = \gamma_E^{(0)}-1$. $E_0$ is similar to
\[(\widehat A_0^{-1/2} \widehat S_0 \widehat A_0^{-1/2})^{-1} = (I + \delta^{-1} A_0^{-1/2}{B_1^\top B_1}A_0^{-1/2}
+ \delta^{-1} A_0^{-1/2}{B_2^\top B_2}A_0^{-1/2})^{-1} \equiv (I + \delta^{-1} W)^{-1}=\delta(W + \delta I)^{-1}.\]
Then \[ \gamma_E^{(0)} = \delta \frac{w^\top (W+ \delta I)^{-1} w}{w^\top w} =
\delta \frac{z^\top  z}{z^\top W z + \delta z^\top z} = \delta \left(\frac{z^\top W z}{z^\top z} + \delta\right)^{-1}
= O(\delta).\]    
We conclude that the preconditioned matrix has one positive eigenvalue at one and negative eigenvalues characterized 
as $O(\delta)-1$.
\subsection{Numerical example} 

\noindent
We notice that $\boldsymbol \alpha_R$ can be arbitrarily small, also in the exact case, depending on the angle
between the subspaces generated by $R_1^\top$ and $R_2^\top$, which influences the behavior of the iterative solvers.
To show this we conducted a simple experiment in solving a linear system of the form \eqref{zeroblocks} with sizes $n_0 = 2000; n_1 = 1200; n_2 = 600$. The coefficient matrix $\mathcal A$ is constructed with $A_0 = I, B_k = \texttt{sprand($n_k, n_0, 0.1$)}$.

\noindent
\textbf{Forward preconditioner}. 
The preconditioner considers:
\begin{enumerate}[(a)]
\item the exact full Schur complements, see \eqref{Sk}: 
 $\mathcal P_D = \text{blkdiag}(I, B_1 B_1^\top, B_2 B_2^{\top} - B_2 B_1^\top (B_1  B_1^\top)^{-1} B_1 B_2^\top)$.
\item the exact incomplete Schur complements $\mathcal P_D = \text{blkdiag}(I, B_1 B_1^\top, B_2 B_2^{\top})$.
 \end{enumerate}
 We found the following eigenvalue distribution of the preconditioned matrix (true eigenvalues and bounds) together with the iterations required by the MINRES solver:
\begin{center}
\begin{tabular}{l|lccl|r||lccl|r}
                &  \multicolumn{4}{c}{Full} &its& \multicolumn{4}{c}{incomplete}& its\\ \hline
Bounds             &  -16.588    &-3.43$\times 10^{-3}$ & 1 & 17.588  & 1248 & -1       & $-1.721 \times 10^{-3}$ & 1 & 2 & 277\\
True &    -16.588    &-3.42$\times 10^{-3}$   &    1&     17.588 &  &  -0.9994     & $-1.718 \times 10^{-3}$    &   1  &     1.9994&\\
\end{tabular}
\end{center}
%If we change the linear system with this symmetrical scaling
%$\mathcal A \longrightarrow T^{1/2} A T^{1/2} $, with $T = text{blkdiag}(I, I, Z) $ and
%$Z = $
%
\smallskip

\noindent
The incomplete preconditioner outperforms its full counterpart. The predicted bounds  perfectly describe
the eigenvalues of the preconditioned matrices.

\noindent
\textbf{Reverse preconditioner}. We computed the eigenvalue distribution for values of $\delta = 10^{-i}, i = 0, 1, 2$, and run the preconditioned MINRES for the same instances with the following outcomes:
\begin{center}
    \begin{tabular}{l|ccc|r}
        $\delta$ &  $\lambda_-^{LB}$ &$\lambda_-^{UB}$ & $1+\lambda_-^{UB}$ & its  \\ \hline
         1    & -0.9999 & -0.1370 & 0.8630 &  14 \\
          0.1  & -1.0000 & -0.6135 & 0.3865 & 7\\
          0.01 & -1.0000 & -0.9407 & 0.0593 & 4\\
        \end{tabular}
\end{center}
which shows that the largest negative eigenvalue approaches $-1$ as $\delta$ becomes small. Apparently, the reverse preconditioner is more convenient than 
the forward instances, since it requires  a considerably smaller number of iterations. 
As well known, however, the eigenvalue distribution (and therefore the predicted number of iterations) is not the only parameter to judge a preconditioner. Other important
issues such that the cost of the preconditioner application and its conditioning 
must be taken into account. This will be clear in \ref{sec:liquid} where we will
address a realistic problem.
% Finally, to highlight the importance of the mutual relation between $R_1 (B_1)$ and $R_2 (B_2)$, we changed the previous setting by modifying $B_2$ as/
%\[ B_2 = B_2 (I - B_1^\top (B_1^\top B_1)^{-1} B_1), \qquad \text{such that} \ B_2 B_1^\top = 0.\]
%We obtained a preconditioned matrix with three distinct eigenvalues $\left\{-0.618, 1, 1.618\right\}$ and
%convergence of MINRES in 3 iterations.
\subsection{Liquid Crystal director modeling}
\label{sec:liquid}
In liquid crystal director modeling, we consider the problem of modeling equilibrium configurations of nematic
liquid crystals \cite{ramage2013preconditioned,BenziLC} in which the unit-vector
constraints on the director field couples with an electric field via the electrostatic potential U . After
discretization and Newton linearization, the double saddle-point coefficient matrix reads as
\[ \mathcal{A} = \begin{bmatrix} A_0 & B_1^\top & B_2^\top \\ B_1 & 0 & 0 \\ B_2 &0 & -A_2\end{bmatrix},\]
with the sizes of the blocks usually satisfy $n_0 = 3n_1; n_2 = n_1$, with the overall size
of the system $m = 5 n_0$. The test case we consider is based on a particular selection of the parameters which we take from \cite{BenziLC}.

\noindent
\textbf{Forward preconditioner}. We first consider  the following forward incomplete Schur complement preconditioner:
\[\mathcal P_D = \text{blkdiag} (A_0, (B_1 A_0 B_1^\top)^{-1}, A_2)\]
The choice of $S_1$ is based on the fact, observed in \cite{BenziLC}, that $B_1 A_0 B_1^\top \approx (B_1 A_0^{-1} B_1^\top)^{-1}$,
$S_2 = A_2$ is motivated by the different order of magnitudes between $A_2$  and $B_2 A_0^{-1} B_2^\top$.
The application of the preconditioner turns out very cheap since it consists in (a) solution with $A_0$, whose exact Cholesky factor dos not introduce any fill-in, (b) multiplication by $B_1 A_0 B_1^\top$, and (c) solution with a tridiagonal matrix $A_2$. What about the eigenvalue distribution? The results that follow refer
to a case in which $n_0 = 3069$ the overall dimension being $N = 5115$.
The indicators are here: $\alpha_E^{(0)} = \beta_E^{(0)} = 1$; $\alpha_E^{(2)} = \beta_E^{(2)} = 1$;
$\alpha_R^{(1)} = 1, \beta_R^{(1)} = 1.394$; 
$\alpha_R^{(2)} = 1.03\times 10^{-7}, \beta_R^{(2)} = 0.175$. 

From these values we have the following bounds
$\lambda \in [-1.177,0] \ \cup [1, 1.8257]$, compared to the true eigenvalue endpoints which are
$  [-1.1514  -0.60184]\cup   [1,   1.8093]$. Three out of four bounds are very close to the exact eigenvalues.

\noindent
\textbf{Reverse preconditioner}. 
Following the discussion in \ref{reverse}, the preconditioner is
\[ \mathcal P_D = \text{blkdiag}(A_0 + \delta ^{-1} B_1 B_1^\top, \delta I , A_2). \]
Application of $S_0$ is the most consuming task in the whole $\mathcal P_D$ application. It can be
carried on by a direct solver or, taking into account the Sherman-Morrison-Woodbury (SMW) formula, by writing
the explicit inverse of $S_0$ as
\[ (A_0 + \delta ^{-1} B_1 B_1^\top)^{-1} = A_0^{-1}\left(I - B_1^\top (\delta I + \underbrace{B_1 A_0^{-1} B_1^\top}_{\Sigma})^{-1} B_1 A_0^{-1} \right)=
A_0^{-1}\left(I - B_1^\top (\delta I  + \Sigma)^{-1} B_1 A_0^{-1} \right). \]
We can perform the following approximation
\[ (\delta I + \Sigma)^{-1} = \Sigma^{-1} (I + \delta \Sigma^{-1})^{-1}\approx
\Sigma^{-1} (I - \delta \Sigma^{-1}),\]
Solution of a system like $(\delta I + \Sigma) v = r$ can therefore be obtained by solving two linear systems with $\Sigma$ as shown in Algorithm \ref{algSigma}
\begin{algorithm}
    \begin{algorithmic}[h!]
        \STATE solve $\Sigma z = r$ by the PCG  with preconditioner $B_1 A_0 B_1^\top$, and tolerance $10^{-14}$.
        \STATE $w = r - \delta z$
        \STATE solve $\Sigma v = w$ by the PCG  with preconditioner $B_1 A_0 B_1^\top$, and tolerance $10^{-14}$.
    \end{algorithmic}
    \caption{Solution of the system $(\delta I + \Sigma) v = r$.}
    \label{algSigma}
\end{algorithm}

\noindent
$\Sigma$ is never formed explicitly, we provided the PCG with a function handle implementing the application of $\Sigma$ to a vector.
We selected $\delta = 10^{-6}$. With this value, the exact eigenvalue intervals are (again for $N = 5155$) $[-1.023, -0.9959] \cup [1, 1.023]$, predicting very
fast convergence of the MINRES method.

The two described preconditioners have been used in combination with the MINRES iterative solver for test cases
with size up to $m = 10\,485\,755$. We fixed the exact solution to be the vector of all ones. We compared these with the block triangular preconditioner (called $\mathcal P_T$)
in \cite{BenziLC}, used to accelerate the (right preconditioned) GMRES solver. This preconditioner is written as
\[ \mathcal P_T = \begin{bmatrix} A_0 & B_1^\top & B_2^\top \\ \gamma B_1 & 0 & 0 \\ 0 & 0  & -A_2\end{bmatrix},
\] and the subsystem $\begin{bmatrix} A_0 &  B_1^\top\\ \gamma B_1 & 0 \end{bmatrix} 
\begin{bmatrix} w_1 \\ w_2\end{bmatrix} = \begin{bmatrix} r_1 \\ r_2\end{bmatrix}$ solved inexactly in
the preconditioner application. To make the correct comparison with our solver, we slightly changed the implementation of this task (thus improving its efficiency), in the solution of the Schur complement $B_1 A_0^{-1} B_1^\top$ which, in \cite{BenziLC} is obtained by the PCG preconditioned with $B_1A_0 B_1^\top$. Here we use the latter matrix to directly approximate the inverse of the Schur complement, as described in Algorithm \ref{alg1}.
\begin{algorithm}
\begin{algorithmic}[h!] 
\STATE Solve  $A_0 z = r_1$
\STATE Compute $w_2  = (B_1 A_0 B_1^\top) \left(B_1 z - \frac{1}{\gamma} r_2\right)$
\STATE Solve $A_0 w_1 = r_1 - B_1^\top w_2 $
\end{algorithmic}
\caption{Solution of the $2 \times 2$ subsystem.}
\label{alg1}
\end{algorithm}

\noindent
This choice allows the use of GMRES (instead of Flexible GMRES).
 
 In both cases (MINRES/GMRES) the exit test is achieved when the relative residual norm is below $10^{-12}$.
The results are summarized in Table \ref{tabris}, where we also report the value of the logarithm (in base 10) of the \textit{a posteriori} relative error. Regarding the block triangular preconditioner we used $\gamma =1$, which is the value suggested in \cite{BenziLC}. 
The numerical experiments have been conducted by a MATLAB (R2025b) implementation run on an  8-core model: Intel Core Ultra 7 258V at 4.8 GHz with 32 GB RAM.
\begin{table}[h!]
    \begin{tabular}{r|rrr|rrr|rrr|rrr|}
    & \multicolumn{9}{c}{MINRES} & \multicolumn{3}{|c}{GMRES} \\
&  \multicolumn{3}{c}{$\mathcal P_D^{\text{rev}} - S_0$} &\multicolumn{3}{c}{$\mathcal P_D^{\text{rev}} - $ SMW} &  
\multicolumn{3}{c}{$\mathcal P_D$} &
\multicolumn{3}{|c}{$\mathcal P_T$} \\
    $m$ & its & CPU & err  &  its& CPU & err  &its  & CPU & err  & its & CPU & err   \\
    \hline
    % 5115 &   11 &   0.01 &  -11 &   10 &   0.03 &   -8 &   26 &   0.00 &  -11 &  12  &   0.01 &  -10 \\ 
    % 10235 &   11 &   0.02 &  -11 &   10 &   0.05 &   -8 &   24 &   0.01 &  -10 &  12  &   0.02 &  -10 \\ 
    % 20475 &   11 &   0.03 &  -10 &   10 &   0.10 &   -8 &   27 &   0.02 &   -9 &  11  &   0.03 &   -9 \\ 
    % 40955 &   11 &   0.06 &  -11 &   10 &   0.21 &   -8 &   28 &   0.04 &   -9 &  15  &   0.05 &   -8 \\ 
     81915 &   11 &   0.14 &   -9 &   10 &   0.48 &   -8 &   30 &   0.07 &   -8 &  10  &   0.09 &   -6 \\ 
    163835 &   11 &   0.28 &   -9 &   12 &   1.18 &   -7 &   30 &   0.14 &   -7 &  10  &   0.19 &   -6 \\ 
    327675 &   11 &   0.60 &   -9 &   12 &   2.37 &   -7 &   30 &   0.28 &   -7 &  10  &   0.38 &   -5 \\ 
    655355 &   11 &   1.22 &   -8 &   16 &   6.69 &   -6 &   32 &   0.57 &   -5 &  10  &   0.72 &   -4 \\ 
   1310715 &   12 &   2.66 &   -9 &   18 &  19.60 &   -7 &   32 &   1.66 &   -5 &  10  &   2.80 &   -4 \\ 
   2621435 &   12 &   9.31 &   -8 &   20 &  37.54 &   -5 &   32 &   2.57 &   -4 &  10  &   3.45 &   -4 \\ 
   5242875 &   13 &  16.25 &   -7 &   \multicolumn{3}{c|}{$\dag$} &  32 &   5.88 &   -4 &   9  &  12.04 &   -2 \\ 
  10485755 &   14 &  40.56 &   -5 &   \multicolumn{3}{c|}{$\dag$}&   33 &  15.01 &   -2 &   9  &  28.18 &   -2 \\ 
    \end{tabular}
\caption{Number of iterations, CPU time and logarithm of the error norm for the described preconditioners in combination with MINRES and GMRES for increasing sizes. \\
$\dag$ means no convergence within 200 seconds.}
\label{tabris}

\end{table}

MINRES accelerated with $\mathcal P_D$ reveals the most efficient solver, even though  the number of iterations 
is not the smallest one. This is mainly due to the extremely cheap preconditioner, which requires only
one solution with $A_0$ at each iteration. By contrast, GMRES-$\mathcal P_T$ needs two solves with $A_0$ (see Algorithm \ref{alg1}) in addition
to the higher complexity of the GMRES solver. The reverse preconditioner (direct solve with $S_0$) takes a smaller number of iterations but it is inefficient due to the cost of the solution of a linear system with matrix $S_0$. Finally its SMW-based implementation reveals
impractical, mostly due to the ill-conditioning of $S_0$. Surprisingly, the most accurate combination
is given by the MINRES with the reverse block diagonal preconditioner, which provides an error norm of several orders of magnitude smaller than that provided by the other solvers.

%Finally, comparing the last columns in Table \ref{tabris} with  FGMRES ($\gamma=1$) in \cite[Table 8]{BenziLC},
%%he claim thar our implementation of $\mathcal P_T$ is more efficient is confirmed: in fact, considering
%5e.g. the case with $N=327675$ we employ $10 \times 2 = 20$ solves with $A$ compared with $4 \times 10 = 40$ solves
%(4 outer iterations with 10 PCG iterations each), the CPU times being 0.38 vs 0.91 seconds.

  \section{Downright-arrow structure}
\label{sec:downright}
Some multiple saddle-point matrices have the following structure
\[ \mathcal A = \begin{bmatrix} A_0 & && B_1^\top \\[.3em] & A_1 &   & B_2^\top\\  & & \ddots &  & \\B_1 &
B_2  & \ldots  & -A_N
\end{bmatrix} = \begin{bmatrix} \bA_- &  \widetilde{\bB}^\top\\ \widetilde{\bB} & - A_N\end{bmatrix},\qquad \begin{array}{lcl} 
\bA_-  &=& \text{blkdiag}(A_0, A_1, \ldots, A_{N-1}), \\[.4em] \widetilde{\bB} &=& \begin{bmatrix} B_1 & B_2 & \ldots & B_N \end{bmatrix} \end{array}.\]
Again, we need that $A_0$ is SPD, $A_k$ are symmetric positive semi-definite and $B_k$ have full row rank. 
In this framework, the usual constraint on the dimensions is $n_k \le n_{k-1}, \ k \ge 1$, and
Assumptions \ref{ass1} and \ref{ass2} are no longer required. The Schur complements
are, in this case:
\[ S_k = A_k, \ k = 0, \ldots, N-1, \quad  S_N  = A_N + \sum_{k=1}^N B_k A_k^{-1} B_k^\top
= A_N + \bBt \bA_-^{-1} \bBt^\top.\]
As also observed in \cite{BradleyGreif2026} for double saddle-point matrices, this matrix can be symmetrically
partitioned into a block tridiagonal one, whose eigenvalue analysis is well understood, see e.g.  \cite{bergamaschi2025eigenvalue, pilotto-et-al2026}. However, 
in the line of Section 4.2, we can address the linear system without permutation, the eigenvalues of the preconditioned matrix being as those of $\mathcal Q$
where
\begin{equation}
\label{Qdown}
    \begin{aligned}
    \mathcal{Q}&= {\mathcal{P}}_D^{-1/2} \mathcal{A} {\mathcal{P}}_D^{-1/2}=\left[\begin{array}{ccccc}
        \bE_- & \bRt^\top\\
        \bRt & - E_N\\
    \end{array}\right], \qquad \mathcal P_D = \begin{bmatrix} \mathbf{\widehat {A}}_- & 0 \\ 0 & \widehat S_N\end{bmatrix}, \quad \bRt = \begin{bmatrix} R_1 & R_2 & \ldots & R_N \end{bmatrix}
    \end{aligned}
\end{equation}
and
\begin{eqnarray*}
\bE_- = \mathbf{\widehat A}_-^{-1/2} \mathbf{A}_- \mathbf{\widehat A}_-^{-1/2}, && \widehat S_N = A_N +\bBt \mathbf  A_-^{-1} \bBt^\top, \quad E_N = \widehat S_N^{-1/2}  A_N \widehat S_N^{-1/2}\qquad \bRt = \widehat S_N^{-1/2} \bBt \mathbf{\widehat{A}}_-^{-1/2}, \quad \\
\end{eqnarray*}
from which \[ \sum_{k=1}^N R_k R_k^\top = \bRt \bRt^\top  = \widehat S_N^{-1/2} \bBt\mathbf{\widehat{A}}_-^{-1} \bBt^{\top} \widehat S_N^{-1/2} = \widehat S_N^{-1/2}(\widehat S_N - A_N)\widehat S_N^{-1/2}
= I - E_N.\]

 %    \subsection{Eigenvalues of the downright arrow matrices}
We wish to characterize the eigenvalues of \eqref{Qdown}:
\[    
    \begin{bmatrix}
                \bE_- & \mathbf{\widetilde{R}}^\top\\
        \mathbf{\widetilde{R}} & - E_N\\
    \end{bmatrix} \begin{bmatrix} \boldsymbol u_- \\ u_N \end{bmatrix}
=   \lambda\begin{bmatrix} \boldsymbol u_- \\ u_N \end{bmatrix}.
\] 
Let $\lambda \not \in [\boldsymbol \alpha_E, \boldsymbol \beta_E]$, then $\bE_- - \lambda \boldsymbol I$
is definite and hence invertible. We can write
\[ \boldsymbol u_- = (\lambda \mathbf I - \mathbf E_-)^{-1} \mathbf {\widetilde{R}}^\top u_N. \]
Upon substitution, the first equation reads
\[(  E_N + \lambda I - \mathbf {\widetilde{R}}(\lambda \mathbf I - \mathbf E_-)^{-1} \mathbf {\widetilde{R}}^\top)u_N = 0 .\]
Premultiplying by $\dfrac{u_N^T}{u_N^T u_N}$, defining $s_k = R_k^\top u_N$, and applying Lemma \ref{Lemma1} to $\dfrac{s_k^\top (\lambda I - E_{k-1})^{-1}s_k}{s_k^\top s_k}$ yield:
\begin{equation} \label{downright}
0 = \frac{u_N^\top E_N u_N}{u_N^\top u_N} + \lambda - \sum_{k=1}^N \frac{1}{\lambda - \dfrac{v_k^\top E_{k-1} v_k}{v_k^\top v_k}} \frac{u_N^\top R_k R_k^\top u_N}{u_N^\top u_N} = 
\lambda + \gamma_E^{(N)} - \sum_{k=1}^{N}\frac{\gamma_R^{(k)}}{\lambda- \gamma_E^{(k-1)}} \equiv \Phi(\lambda).\end{equation}
The function $\Phi: (-\infty, \boldsymbol \alpha_E) \cup (\boldsymbol \beta_E, + \infty) \to \mathbb R$ is derived from $\Pi$ \eqref{Pidef} by considering the change of variable $\gamma_E^{(j)} \to -\gamma_E^{(N-j)}$.
\subsection{PDE-constrained optimization example}
We now test the results of Section \ref{sec:downright} on the following PDE-constrained optimization example (see work such as \cite[Sec. 5]{Bergamaschi-COAP-2024}, \cite{MNN}, and \cite[Sec. 4.1]{pearson2023symmetric}):
\begin{equation*}
\min_{y,u} ~~ \frac{1}{2}\left\|y-\widehat{y}\right\|_{L^2(\partial\Omega)}^2+\frac{\beta}{2}\left\|u\right\|_{L^2(\Omega)}^2 \qquad \text{s.t.} \quad \left\{\begin{array}{rl}
-\Delta y+y+u=0 & \text{in }\Omega, \\
\frac{\partial y}{\partial n}=0 & \text{on }\partial\Omega. \\
\end{array}\right.
\end{equation*}
Here, the domain $\Omega := (0,1)^2$ with boundary $\partial\Omega$, $y$ and $u$ denote \emph{state} and \emph{control variables}, and $\beta$ is a (positive) regularization parameter. The function $\widehat{y}$ denotes a (given) \emph{desired state}, which for this example is obtained by solving the forward problem with control $4x_1(1-x_1)+x_2$, with $x_1$ and $x_2$ denoting the spatial coordinates.

Discretizing this problem using $P_1$ finite element basis functions may be cast as a double saddle-point linear system with downright-arrow structure:
\begin{equation}\label{PDECO_system}
\left[\begin{array}{ccc}
\beta M & 0 & M \\ 0 & M & L \\ M & L & 0 \\
\end{array}\right]\left[\begin{array}{c}
u_h \\ y_h \\ p_h \\
\end{array}\right]=\left[\begin{array}{c}
0 \\ \widehat{y}_h \\ 0 \\
\end{array}\right],
\end{equation}
where $y_h$, $u_h$, $p_h$, and $\widehat{y}_h$ denote the discretized state, control, \emph{adjoint}, and desired state, $M$ denotes a finite element mass matrix, and $L$ is the sum of finite element stiffness and mass matrices. Denoting with $n_0$ the size of both $M$ and $L$, the overall size of the linear system
is $m = 3 n_0$. In this experiment, we precondition \eqref{PDECO_system} with
%\textcolor{blue}{LB: I notice that the preconditioner in this test case is lower triangular, while it is upper triangular in our settings}
\begin{equation*}
\mathcal{P} = \left[\begin{array}{ccc}
\beta \widehat{M} & 0& 0 \\ 0 & \widehat{M} & 0 \\ 0 & 0 & \widehat{S}_2\\
\end{array}\right],
\end{equation*}
where $\widehat{M}$ denotes the lumped mass matrix $M$, which provides
$[\alpha_E^{(0)},\beta_E^{(0)}] \in [0.25,1]$ \cite{wathen87}, and $\widehat{S_2}$ is aimed at approximating
the exact Schur complement $ S_2 = \dfrac{1}{\beta }M + {L} M^{-1} L $.
To approximate $S_2$, we take
\begin{equation*}
\widehat{S}_2 = \frac{3}{4} \left( \frac{1}{\sqrt{\beta}} M  + L \right)_{\text{AMG}} \, M^{-1} \, \left( \frac{1}{\sqrt{\beta}} M + L\right)_{\text{AMG}},
\end{equation*} 
where $(\cdot)_{\text{AMG}}$ denotes the application of 2 V-cycles of the \texttt{HSL\_MI20} algebraic multigrid solver \cite{HSL_MI20,HSL_MI20_code} to a given matrix, with 2 symmetric Gauss--Seidel iterations serving as a pre- and post-smoother. The constant $\frac{3}{4}$ is included to ensure the eigenvalues of the preconditioned matrix $X$ (if the $L + \beta^{-1/2} M$ matrices were approximated exactly) are contained within the interval $[\frac{2}{3}, \frac{4}{3}]$, that is symmetrically distributed about 1 (see \cite{PW12} for the result without this factor).
We can evaluate $\alpha_R^{(j)}, \beta_R^{(j)}, j = 1,2$ observing that
\begin{equation*}
	\begin{array}{lcllcl}
		R_1 & = & \widehat S_2^{-1/2} M \widehat{M}^{-1/2}/\sqrt{\beta},  &R_1 R_1^\top &=&\widehat S_2^{-1/2} M \widehat{M}^{-1} M \widehat S_2^{-1/2}/\beta \\ [.4em]
		R_2 & =&  \widehat S_2^{-1/2} L \widehat{M}^{-1/2}, &R_2 R_2^\top &=&\widehat S_2^{-1/2} L \widehat{M}^{-1} L \widehat S_2^{-1/2}  \\ [.4em]
		R_1 R_1^\top + R_2 R_2^\top &=& \multicolumn{4}{l}{\widehat S_2^{-1/2}M\widehat{M}^{-1} (\underbrace{M/\beta  + \widehat{M} M^{-1} L \widehat{M}^{-1} L}_{\text{\normalsize{$T$}}})\widehat S_2^{-1/2}}.
\end{array}
\end{equation*}
%Now \begin{eqnarray*}
	%R_1 R_1^\top + R_2 R_2^\top &=& \widehat S_2^{-1/2}M\widehat{M}^{-1} (\underbrace{M/\beta  + \widehat{M} M^{-1} L \widehat{M}^{-1} L}_{\text{\normalsize{$T$}}})\widehat S_2^{-1/2}.% =\\ &&\widehat S_2^{-1/2}M\widehat{M}^{-1} S_2^{-1} (\underbrace{M/\beta  +  \widehat{M} M^{-1} L \widehat{M}^{-1} L}_{T}) S_2 \widehat S_2^{-1/2}
 %\end{eqnarray*}
 We may notice that the matrix $T$ is $M/\beta $ (first addend in $S_2$) 
 plus a matrix which resembles the second addend of $S_2$. We can then approximate $T \approx S_2$
 and write
 \[ R_1 R_1^\top + R_2 R_2^\top \approx \widehat S_2^{-1/2}M\widehat{M}^{-1}  S_2 \widehat S_2^{-1/2},\]
the right-hand matrix being similar to $X = \widehat S_2^{-1} S_2 M\widehat{M}^{-1}$.
We can bound the eigenvalues of $X$ considering that  $\widehat S_2$ is obtained combining two approximations: the one coming from the matching strategy and  the other from the AMG preconditioning which clusters the eigenvalues of the preconditioned matrix in (roughly) $[0.9,1]$. Summarizing
\begin{eqnarray*}
\lambda_{\min}( \widehat S_2^{-1} S_2) \ge \frac 23 \times 0.9^2 = 0.54, && 
\lambda_{\max}( \widehat S_2^{-1} S_2) \le \dfrac 43 \times 1^2 = \dfrac 43;  \\
\lambda_{\min}( \widehat M^{-1} M) = 0.25, && 
\lambda_{\max}( \widehat M^{-1} M) =  1; \\
	 \lambda_{\min}(X) \ge  0.54 \cdot 0.25 = 0.135  \equiv \alpha_R^{(1)} + \alpha_R^{(2)}& &
 \lambda_{\max}(X) \le  \frac 43 \equiv   \beta_R^{(1)} + \beta_R^{(2)}
\end{eqnarray*}
%Finally $\lambda_{\min}(X) \ge  0.135 \equiv \alpha_R^{(1)} + \alpha_R^{(2)}
%, \ \ \lambda_{\max}(X) \le  \dfrac 43 \equiv \beta_R^{(1)} + \beta_R^{(2)}.$
Turning now to \eqref{downright}, this becomes
\[ \Phi(\lambda) = \lambda 
- \frac{\gamma_R^{(1)}}{\lambda - \gamma_E^{(0)}} 
- \frac{\gamma_R^{(2)}}{\lambda - \gamma_E^{(1)}}=0. 
\]
This function in increasing in $(-\infty,\alpha_E^{(1)}) \cup (\beta_E^{(1)}, +\infty)$ and satisfies $\Phi(0) > 0$.
The smallest eigenvalues is bounded by the negative zero of $\Phi$ with parameters $\beta_R^{(j)}$ and $\alpha_E^{(j)}$. Since $\alpha_E^{(0)} = \alpha_E^{(1)}$ the equation simplifies to
\[ \lambda^2 - \lambda \alpha_E^{(0)} - \beta_R^{(1)} - \beta_R^{(2)} = \lambda^2 - 0.25 \lambda - \frac 43 = 0, 
 \ \Longrightarrow \ \lambda_- = -1.0363.\]  The equation for the largest negative roots is
\[ \lambda^2 - \lambda \beta_E^{(0)} - \alpha_R^{(1)} - \alpha_R^{(2)} = \lambda^2 - \lambda - 0.135 = 0, 
 \ \Longrightarrow \ \lambda_- = -0.1205.\] 
 The smallest positive eigenvalue is bounded by $\alpha_E^{(0)} = 0.25$. Finally, for the largest positive eigenvalues
 we solve 
 \[ \lambda^2 - \lambda \beta_E^{(0)} - \beta_R^{(1)} - \beta_R^{(2)} = \lambda^2 - \lambda - \frac 43 = 0, 
 \ \Longrightarrow \ \lambda_+ = 1.7582.\] We now compare the bounds with the true computed eigenvalues
\begin{center}
 \begin{tabular}{lrrrr}
 Bounds                    & --1.0363   &   --0.1205   &    0.2500 &   1.7582\\ \hline
True ($\beta = 1$)         & --0.7166   &   --0.3985   &    0.2500 &   1.6641\\ 
True ($\beta = 10^{-3}$)   &  -0.7003   &   --0.3908   &    0.2500 &   1.6524
\end{tabular}
\end{center}
We can appreciate the adherence of the four bounds to the exact endpoints of the two intervals.
We finally report in Table \ref{tabPDE} the results in terms of number of iterations and CPU times of the MINRES solver
equipped with the block diagonal preconditioner for $\beta = 1, 10^{-3}$. The iteration is stopped whenever the relative residual norm is below $10^{-8}$.
\begin{table}[h!]
    \begin{center}
\begin{tabular}{r|rr|rr}
   & \multicolumn{2}{|c|}{$\beta = 1$} 
   & \multicolumn{2}{c}{$\beta = 10^{-3}$} \\
$m$ & its & CPU & its & CPU \\ \hline
   3267 &   55&  0.03 & 70&  0.03 \\
   12675&   57&  0.27 & 75&  0.21 \\  
   49923&   60&  0.48 & 78&  0.73 \\  
  198147&   65&  1.91 & 81&  2.42 \\  
  789507&   67&  8.13 & 86& 10.47 \\ 
 3151875&   71& 36.35 & 93& 47.71 \\ 
\end{tabular}
\end{center}
\caption{Number of iterations and CPU times in the solution of the PDE-constrained optimization problem with MINRES and the block
diagonal preconditioner.}
\label{tabPDE}
\end{table}

The almost perfect scalability of the preconditioned solver is appreciated by the very slight increase of the
number of iterations and, even more importantly, by the linear proportionality of the CPU times with the size of the problem.

\section{Conclusions}
\label{sec:conc}
We have characterized the Schur complement blocks in multiple saddle-point matrices having an arrow  structure. The eigenvalues of such (block-diagonal preconditioned) matrices are shown to coincide with the zeros of suitable parametric rational functions,
whose complexity does not grow significantly with the number of blocks.
Simple evaluations of such zeros for particular choices of the parameters allow us to characterize the bounds of the two intervals (one negative and one positive) containing
the eigenvalues of the preconditioned matrices.

An important outcome of this study is that using the full Schur complement block diagonal preconditioner does not represent the optimal choice, not only because of the
extreme complexity of their definition (which could be hardly implemented in realistic settings), but also in terms of eigenvalue distribution and number of iterations. We proved that a straightforward approximation of the Schur complements produces a better preconditioner both in terms of application cost and of number of iterations.

Numerical results onto both synthetic and realistic (large size) test cases taken from the liquid crystal director modeling and a PDE-constrained optimization problem, show
the accuracy of the developed bounds, which are able to predict (in most cases) 
the behavior of the preconditioned iterative solver (MINRES in our case).

Future work aims at analyzing block triangular preconditioners for multiple saddle-point arrow matrices. This will pose a number of issues such that, among the others,
(a) characterizing complex eigenvalues (in view of the lost symmetry of the preconditioned matrix), (b) which blocks are worth neglecting in the preconditioner definition.

\end{document}